\documentclass[10pt,reqno]{amsart}
\usepackage{eucal,fullpage,times,amsmath,amsthm,amssymb,mathrsfs,stmaryrd,color,enumerate,accents}
\usepackage[all]{xy}
\usepackage{url}

\newcommand{\calO}{{\mathcal{O}}}

\newcommand{\calQ}{\mathcal{Q}}

\newcommand{\G}{\mathbf{G}}
\newcommand{\Z}{\mathbf{Z}}

\newcommand{\F}{\mathbf{F}}
\newcommand{\Q}{\mathbf{Q}}
\renewcommand{\P}{\mathbf{P}}

\newcommand{\Spec}{{\mathrm{Spec}}}

\newcommand{\Hom}{\mathrm{Hom}}

\newcommand{\Shv}{\mathrm{Shv}}
\newcommand{\GL}{\mathrm{GL}}
\newcommand{\Ext}{\mathrm{Ext}}

\newcommand{\Vect}{\mathrm{Vect}}

\newcommand{\Mod}{\mathrm{Mod}}

\newcommand{\R}{\mathrm{R}}

\newcommand{\Pic}{\mathrm{Pic}}
\newcommand{\et}{\mathrm{\acute{e}t}}

\newcommand{\id}{\mathrm{id}}

\newcommand{\Tr}{\mathrm{Tr}}

\newcommand{\rig}{\mathrm{rig}}

\renewcommand{\ker}{\mathrm{ker}}

\newcommand{\fram}{\mathfrak{m}}

\newcommand{\comment}[1]{}

\newcommand{\cosimp}[3]{\xymatrix@1{#1 \ar@<.4ex>[r] \ar@<-.4ex>[r] & {\ }#2 \ar@<0.8ex>[r] \ar[r] \ar@<-.8ex>[r] & {\ } #3 \ar@<1.2ex>[r] \ar@<.4ex>[r] \ar@<-.4ex>[r] \ar@<-1.2ex>[r] & \cdots }}
\newcommand{\colim}{\mathop{\mathrm{colim}}}

\begin{document}
\bibliographystyle{alpha}
\newtheorem{theorem}{Theorem}[section]
\newtheorem*{theorem*}{Theorem}
\newtheorem*{condition*}{Condition}
\newtheorem*{definition*}{Definition}
\newtheorem*{corollary*}{Corollary}
\newtheorem{proposition}[theorem]{Proposition}
\newtheorem{lemma}[theorem]{Lemma}
\newtheorem{corollary}[theorem]{Corollary}
\newtheorem{claim}[theorem]{Claim}

\theoremstyle{definition}
\newtheorem{definition}[theorem]{Definition}
\newtheorem{question}[theorem]{Question}
\newtheorem{remark}[theorem]{Remark}
\newtheorem{guess}[theorem]{Guess}
\newtheorem{example}[theorem]{Example}
\newtheorem{condition}[theorem]{Condition}
\newtheorem{warning}[theorem]{Warning}
\newtheorem{notation}[theorem]{Notation}
\newtheorem{construction}[theorem]{Construction}

\title{A note on the non-existence of small Cohen-Macaulay algebras}
\begin{abstract}
By finding a $p$-adic obstruction, we construct many examples of complete noetherian local normal $\F_p$-algebras $R$ such that no module-finite extension $R \hookrightarrow S$ is Cohen-Macaulay. These examples should be contrasted with a result of Hochster-Huneke: the directed union of all such extensions is always Cohen-Macaulay.
\end{abstract}
\author{Bhargav Bhatt}
\address{Department of Mathematics \\ University of Michigan \\ Ann Arbor 48109 \\ USA}
\email{bhattb@umich.edu}
\maketitle

\section{Introduction}
\label{sec:intro}

A noetherian ring is called Cohen-Macaulay (CM) if there are no non-trivial relations between elements of a system of parameters. These rings form an exceptionally well behaved class: local cohomology vanishes when it can, and duality works out beautifully. This note studies rings realisable as subrings of CM rings:

\begin{definition}
A noetherian ring $R$ admits a {\em big CM algebra} if there is an injective map $R \hookrightarrow S$ of rings with $S$ CM. If the ring $S$ can be chosen to be finitely generated as an $R$-module, then we say that $R$ admits a {\em small CM algebra}.
\end{definition}

Under mild hypotheses, a fundamental result of Hochster and Huneke \cite{HHBigCM} shows that any ring $R$ that contains a field admits a big CM algebra; in fact, if $R$ is an $\F_p$-algebra, the absolute integral closure of $R$ does the job.  Thus, one asks: does a ring $R$ always admit a small CM algebra? The answer is ``no'' in characteristic $0$ based on a local cohomological obstruction\footnote{A normal local noetherian $\Q$-algebra $R$ that admits a small CM algebra $S$ is itself CM: $H^i_\fram(R)$ is a summand of $H^i_\fram(S)$ via the trace splitting.}. In characteristic $p$, however, the results of {\em loc.\ cit.} immediately nullify any {\em coherent} cohomological obstructions: any non-trivial relation can be trivialised after a finite extension. In fact, by \cite{GLBigCM}, there is a single finite extension $R \hookrightarrow S$ trivialising {\em all} unwanted relations. Nevertheless, our goal in this note is to discuss a negative answer to the above question in characteristic $p$ using a {\em $p$-adic} (rigid) cohomological obstruction.

\begin{theorem}
	\label{thm:mainthm}
	Let $(A,L)$ be polarised projective variety over a perfect field $k$; set $\widehat{R}$ to be completion at the origin of $\oplus_{n \geq 0} H^0(A,L^n)$.  Assume $H^i_\rig(A)_{< 1} \neq 0$ for some $0 < i < \dim(A)$. Then $\widehat{R}$ does not admit a small CM algebra.
\end{theorem}

The hypothesis on $A$ is satisfied, for example, if $A$ is an abelian surface. This theorem generalises a calculation of Sannai-Singh \cite[Example 5.3]{SinghSannai} (who showed the non-existence of a {\em graded} small CM algebra for a specific $A$). This theorem should also be contrasted with a result of Hartshorne: if $A$ is CM, then $R$ admits a small CM {\em module}, i.e., a module whose depth is $\dim(A)$ (see \cite{HochsterSmallCM}). We end by noting that the proof of Theorem \ref{thm:mainthm} shows a stronger statement: any local $\F_p$-algebra that admits a small CM algebra is ``Witt Cohen-Macaulay,'' see Remark \ref{rmk:moregeneralstatement}.

\subsection*{A summary of the proof} Let us informally explain the proof in an example: $A$ is an ordinary abelian surface. The key idea is to work $p$-adically instead of modulo $p$ until the end to track the divisibility properties of local cohomology under finite extensions. More precisely, the $p$-rank of $A$ contributes free $\Z_p$-summands to $H^2_\fram(\Spec(\widehat{R})_\et,\Z_p)$. Trace formalism then shows: for any finite extension $\widehat{R} \to \widehat{S}$, the group $H^2_{\fram}(\Spec(\widehat{S})_\et,\Z_p)$  also contains free $\Z_p$-summands. Reducing these modulo $p$ and using the Artin-Schreier sequence shows $H^2_\fram(\widehat{S}) \neq 0$, so $\widehat{S}$ is not CM. In the body of the note, we work with Witt-vector cohomology instead of $p$-adic \'etale cohomology to apply the preceding argument with fewer ordinarity constaints,  see Remark \ref{rmk:wittversusetale}.

\subsection*{Notation} 
For any $\F_p$-scheme $X$, we write $\{W_n \calO_X\}$ of the projective system of (Zariski) sheaves defined by the truncated Witt vector functors (see \cite{IllusieWitt,SerreWitt}), and $W\calO_X := \lim_n W_n\calO_X$ for its inverse limit (which is also the derived limit: the transition maps $W_n \calO_X \to W_{n-1} \calO_X$ are surjective for all $n$, and $W_n\calO_X$ has no higher cohomology on affines).  For a closed subset $Z \subset X$ and any abelian sheaf $A$, we write $\R\Gamma_Z(X,A)$ for the homotopy-kernel of $\R\Gamma(X,A) \to \R\Gamma(X-Z,A)$. In particular, since cohomology commutes with derived limits, we have $\R\Gamma_Z(X,W\calO_X) \simeq \R\lim_n \R\Gamma_Z(X,W_n\calO_X)$. Recall that  $\R\Gamma(X,-)$ commutes with filtered colimits of sheaves if $X$ is quasi-compact and quasi-separated. In particular, if both $X$ and $X-Z$ are quasi-compact and quasi-separated, then $\R\Gamma_Z(X,W\calO_{X,\Q}) \simeq \big(\R\lim_n \R\Gamma_Z(X,W_n\calO_X)\big) \otimes_{\Z} \Q$.

\subsection*{Acknowledgements} I am very grateful Manuel Blickle and Kevin Tucker for making me aware of the question studied in this note, and several related discussions. Thanks are due to Johan de Jong, Mel Hochster and Anurag Singh for their extremely helpful comments, questions, and encouragement.

\section{Proof}

\subsection{Remarks on Witt vectors}

We use $k$ to denote a fixed perfect base field of characteristic $p$. We start by recalling a fundamental result identifying the slope $< 1$ part of rigid cohomology with Witt vector cohomology; the smooth case is due to Bloch-Deligne-Illusie (see \cite[Theorem 0.2]{BlochCrysKtheory} and \cite[\S II.3]{IllusieWitt}), while the singular case is newer:

\begin{theorem}[{\cite[Theorem 1.1]{BBEWitt}}]
	\label{thm:bbe}
Let $A$ be a proper $k$-variety. Then $H^i(A,W\calO_A)_{\Q} \simeq H^i_\rig(A)_{< 1}$.
\end{theorem}

\begin{remark}
	\label{rmk:wittversusetale}
	Theorem \ref{thm:bbe} is the main reason we work with Witt-vector cohomology instead of $p$-adic \'etale cohomology in this note: the latter only describes the slope $0$ part of rigid cohomology, while the former describes the (potentially much larger) slope $< 1$ part. In particular, if $A$ is an abelian variety, then $H^1(A,W\calO_{A,\Q})$ is always non-zero for weight reasons, while $H^1(A,\Z_p)$ vanishes if $A$ is supersingular and $k = \overline{k}$.
\end{remark}

The next lemma allows deduction of modulo $p$ consequences from rational assumptions.

\begin{lemma}
	\label{lem:devissagewitt}
	Let $Y$ be a $k$-scheme with a closed subscheme $Z$. If $H^i_Z(Y,W\calO_Y)_{\Q} \neq 0$ for some $i$, then $H^j_Z(Y,\calO_Y) \neq 0$ for $j = i$ or $j = i-1$.
\end{lemma}
\begin{proof}
The assumption on $H^i_Z(Y,W\calO_Y)_{\Q}$ and the formula
	\[ \R\Gamma_Z(Y,W\calO_Y)_{\Q} \simeq \big(\R\lim_n \R\Gamma_Z(Y,W_n\calO_Y)\big) \otimes_{\Z} \Q\]
	show that $H^j_Z(Y,W_n\calO_Y) \neq 0$ for some $n > 0$ and some $j \in \{i-1,i\}$ (as $\R^i \lim = 0$ for $i > 1$). The rest follows by standard exact sequences expressing $W_n\calO_Y$ as an iterated extension of copies of $\calO_Y$.
\end{proof}

We need trace maps in Witt vector cohomology, so we recall a direct construction (essentially due to \cite{SuslinVoevodskySingHom}).

\begin{lemma}
\label{lem:trace}
Let $f:Y \to X$ be a finite surjective morphism of noetherian normal schemes. For any abelian sheaf $A$ on $X_\et$ that is representable by an algebraic space, there is a functorial trace map $\Tr:f_*f^*A \to A$ such that the composite $A \stackrel{f^*}{\to} f_* f^* A \stackrel{\Tr}{\to} A$ is multiplication by the generic degree of $f$.
\end{lemma}
\begin{proof}
	We refer the reader to \cite[Proposition 6.2]{Bhattanngrpsch} for a proof.
\end{proof}

\begin{remark}
The trace constructed in Lemma \ref{lem:trace} is non-standard and slightly ad hoc. For example, if $f:X \to X$ is the Frobenius map, then $f^*:\Shv(X_\et) \to \Shv(X_\et)$ is an equivalence, so the ``correct'' trace map should be an equivalence, while the one from Lemma \ref{lem:trace} is multiplication by $\deg(f)$ (composed with the inverse of $\id \stackrel{\simeq}{\to} f_* f^*$). In particular, the trace map constructed in Lemma \ref{lem:trace} does {\em not} furnish a right adjoint to $f_! \simeq f_*$.
\end{remark}

Using trace maps, we show that (rational) Witt vector cohomology cannot be killed by finite covers.

\begin{corollary}
	\label{cor:rigsummand}
	Let $f:Y \to X$ be a finite surjective morphism of noetherian normal $k$-schemes, and let $Z \subset X$ be a closed subset. Then $H^i_Z(X,W\calO_X)_{\Q} \to H^i_{f^{-1} Z}(Y,W\calO_Y)_{\Q}$ is a direct summand.
\end{corollary}
\begin{proof}
	Let $d$ be the generic degree of $f$. As $W_n(-)$ is representable for each $n > 0$, Lemma \ref{lem:trace} gives maps $f_* W_n \calO_Y \to W_n \calO_X$ whose composition with the pullbacks $W_n \calO_X \to f_* W_n \calO_Y$ is multiplication by $d$. These maps are compatible as $n$ varies, so taking limits (and commuting them with $f_*$) gives a map $f_* W\calO_Y \to W\calO_X$ whose composition with $W\calO_X \to f_* W\calO_Y$ is multiplication by $d$. The claim follows by applying $H^i_Z(X,- \otimes_{\Z} \Q)$.
\end{proof}

\subsection{The main theorem}

To prove Theorem \ref{thm:mainthm}, we first establish some notation.

\begin{notation} Fix a polarised projective variety $(A,L)$ of a perfect characteristic $p$ field $k$; set $R = \oplus_{n \geq 0} H^0(A,L^n)$ to be the section ring with $\fram \subset R$ the homogeneous maximal ideal. Set $X = \Spec(R)$, $Z = \{\fram\} \subset X$ and $U = X - Z$ with $\pi:U \to A$ realising $U$ as the total space of $L^{-1}$ over $A$. Set $\widehat{X} = \Spec(\widehat{R})$, $\widehat{U} = \widehat{X} \times_X U$, and abusively let $Z \subset \widehat{X}$ denote the closed point.
\end{notation}

Next, we record the expected relation between the cohomology of $A$ and $U$:

\begin{lemma}
	\label{lem:cones}
The natural map $W_n\calO_A \to \pi_* W_n\calO_U$ is a direct summand for all $n$. In particular, $H^i(A,W_n \calO_A) \to H^i(U,W_n\calO_U)$ is a direct summand for all $i$ and $n$.
\end{lemma}
\begin{proof}
As $\pi$ is a $\G_m$-torsor, the natural map $\calO_A \to \pi_* \calO_U$ realises the source as the weight $0$ eigenspace of the target.  The rest follows by taking products and observing that $W_n(R) \simeq R^n$ as sets functorially in $R$.
\end{proof}

We can now prove the main theorem:

\begin{proof}[Proof of Theorem \ref{thm:mainthm}]
	As $W_n \calO_X$ is an extension of $\calO_X$ by $W_{n-1} \calO_X$, induction and the affineness of $X$ show $H^i(X,W_n \calO_X) = 0$ for $i > 0$. Standard sequences then identify $H^i(U,W_n\calO_U) \simeq H^{i+1}_Z(X,W_n\calO_X)$ for all $i,n > 0$. Since $H^0(U,W_n\calO_U) = W_n(H^0(U,\calO_U))$,  the system $\{H^0(U,W_n\calO_U)\}$ has no $\lim^1$ (as $W_n(R) \to W_{n-1}(R)$ is surjective for any ring $R$), so 
	\[ H^i(U,W\calO_U)_{\Q} \simeq H^{i+1}_Z(X,W\calO_X)_{\Q} \simeq H^{i+1}_Z(\widehat{X},W\calO_{\widehat{X}})_{\Q}\] 
	for $i > 0$, where the last isomorphism comes from the excision identification $\{R\Gamma_Z(X,W_n\calO_X)\} \simeq \{\R\Gamma_Z(\widehat{X},W_n\calO_{\widehat{X}})\}$ of projective systems. Theorem \ref{thm:bbe} and Lemma \ref{lem:cones} then show $H^i_\rig(A)_{< 1} \simeq H^i(A,W\calO_A)_{\Q}$ is a direct summand of $H^{i+1}_{\widehat{Z}}(\widehat{X},W\calO_{\widehat{X}})_{\Q}$ for $i > 0$. Choose $0 < i < \dim(A)$ such that $H^i_{\rig}(A)_{< 1} \neq 0$. Let $f:\widehat{Y} \to \widehat{X}$ be a finite surjective morphism of noetherian normal schemes; we will show that $\widehat{Y}$ is not CM along $f^{-1}Z$. Corollary \ref{cor:rigsummand} shows $H^{i+1}_{f^{-1} Z}(\widehat{Y},W\calO_{\widehat{Y}})_{\Q} \neq 0$. Lemma \ref{lem:devissagewitt} then gives $H^j_{f^{-1}Z}(\widehat{Y},\calO_{\widehat{Y}}) \neq 0$ for some $j \in \{i,i+1\}$, which proves the claim as $i+1 < \dim(A) + 1 = \dim(\widehat{Y})$.
\end{proof}

\begin{remark}
	As rigid cohomology is a Weil cohomology theory, the hypothesis on $A$ in Theorem \ref{thm:mainthm} may be reformulated topologically, at least when $A$ is smooth and $k$ is finite,  to say: for some $0 < i < \dim(A)$, there is at least one Frobenius eigenvalue on $H^i(A_{\overline{k}},\Q_\ell)$ which is not divisible by $p$ (for some auxilliary prime $\ell$ invertible on $k$). Indeed, the eigenvalues occurring in $H^i_\rig(A)$ coincide with those on $H^i(A_{\overline{k}},\Q_\ell)$ and are algebraic integers.
\end{remark}

\begin{example}
Some elementary examples of projective varieties $A$ to which Theorem \ref{thm:mainthm} applies include: any projective variety of dimension $\geq 2$ dominating a positive dimensional abelian variety, any projective variety of dimension $\geq 3$ dominating a $K3$ surface of finite height, etc.
\end{example}

We give an example showing that the presence of non-trivial middle cohomology of the structure sheaf does {\em not} force the non-existence of a small CM algebra for the section ring, even for smooth projective varieties; this example also shows the necessity of making a {\em $p$-adic} (rather than modulo $p$)  assumption on $A$ in Theorem \ref{thm:mainthm}.

\begin{example}
\label{ex:illusieserre}
Assume $k$ is algebraically closed, and let $X$ be a smooth complete intersection in some $\P^n_k$ with $\dim(X) \geq 2$ with the property that a suitable subgroup $G \simeq \F_p \subset \GL_{n+1}(k)$ preserves $X$ and acts fixed point freely on $X$; such examples were constructed by Serre, see \cite[\S 6]{IllusieFormalGeom}. Let $Y = X/G$ denote the quotient. Then $X \to Y$ is a finite \'etale $G$-torsor, and $Y$ is a smooth projective $k$-variety; explicitly, the line bundle $\calO(1) \in \Pic(X)$ is equivariant for the $G$-action, and hence descends to an ample line bundle $L$ on $Y$. Let $R$ and $S$ denote the section rings of $(Y,L)$ and $(X,\calO(1))$ respectively. Then $R \to S$ is a finite extension. We will show that $S$ is CM (even lci), but $R$ is not. The former follows immediately as $(X,\calO(1))$ is a complete intersection. On the other hand, the Lefschetz formalism (see \cite[Corollary XII.3.5]{SGA2}) shows that $X$ is simply connected. The Leray spectral sequence for $X \to Y$ then yields $\F_p \simeq G \simeq \pi^\et_1(Y) \simeq H^1(Y_\et,\F_p)^\vee$,  and hence $H^1(Y,\calO_Y) \neq 0$ by the Artin-Schreier sequence. The latter group is a direct summand of $H^2_\fram(R)$, so $R$ is not CM (as $\dim(R) = \dim(X) + 1 \geq 3$).
\end{example}

\begin{remark}
	\label{rmk:moregeneralstatement}
Let $(R,\fram)$ be a noetherian local $k$-algebra. The proof of Theorem \ref{thm:mainthm} shows that being ``Witt Cohen-Macaulay'' (i.e., having $H^i_\fram(\Spec(R),W\calO_{R,\Q}) = 0$ for $i < \dim(R)$) is necessary for the existence of small CM algebras. We do not know if it is a sufficient condition. A (weakened) graded analogue asks: any projective variety $X$ with $H^i(X,W\calO_{X,\Q}) = 0$ for $0 < i < \dim(X)$ admits an alteration $\pi:Y \to X$ with $H^i(Y,\calO_Y) = 0$ for $0 < i < \dim(Y)$. The simplest non-trivial instance of this question is when $X = S \times \P^1$ with $S$ a supersingular $K3$ surface, where a  positive answer  is implied by Artin's conjecture on unirationality of supersingular $K3$ surfaces. 
\end{remark}

% \begin{remark}
%	As pointed out by Manuel Blickle, Theorem \ref{thm:mainthm} gives an explicit example of the failure of cofiltered limits to commute with filtered colimits. Indeed, for $R$ as in the theorem, the groups
% \[ \colim_{R \to S} \big(\lim_n H^i_\fram(S,W_n\calO_S)\big) \quad \mathrm{and} \quad \lim_n \big(\colim_{R \to S} H^i_\fram(S,W_n\calO_S)\big) \]
% are different; here the colimit is indexed by the category of finite extensions $R \to S$, while the limit is indexed by $\N$. Indeed, the left hand side is non-zero by Theorem \ref{thm:mainthm}, while the right hand side is $0$ by \cite{HHBigCM}.
% \end{remark}

\bibliography{smallcmmod}

\end{document}